\documentclass{article}
\usepackage[utf8]{inputenc}
\usepackage[margin =2.4cm]{geometry}
\usepackage{amssymb}
\usepackage{amsmath}
\usepackage{amsthm}
\usepackage{enumerate} 
\usepackage{graphicx}
\usepackage{csquotes}
\usepackage{subcaption}
\usepackage{amsthm}
\usepackage{lmodern}

\newtheorem{Theorem}{Theorem}
\newtheorem{Proposition}{Proposition}
\newtheorem{Corollary}{Corollary}
\newtheorem{Lemma}{Lemma}

\title{{\fontsize{25pt}{40pt}\selectfont Considering a Classical Upper Bound \\ on the Frobenius Number}}
    \author{} 
    \date{} 

\begin{document}
\maketitle
\vspace{-14mm}
\begin{center}
    {$^{1}${\fontsize{16pt}{40pt}\selectfont Aled Williams \quad Daiki Haijima}} \\  \vspace{4.0mm}
    $^{1}$Department of Mathematics \\
    London School of Economics and Political Science \\
    London, UK \\
    $^{1}$\texttt{a.e.williams1@lse.ac.uk} \\
    \vspace{5mm} 
\end{center}



\begin{abstract}
In this paper we study the (classical) Frobenius problem, namely the problem of finding the largest integer that cannot be represented as a nonnegative integral combination of given relatively prime (strictly) positive integers (known as the Frobenius number). The main contribution of this paper are observations regarding a previously known upper bound on the Frobenius number where, in particular, we observe that a previously presented argument features a subtle error, which alters the value of the upper bound. Despite this, we demonstrate that the subtle error does not impact upon on the validity of the upper bound, although it does impact on the upper bounds tightness. Notably, we formally state the corrected result and additionally compare the relative tightness of the corrected upper bound with the original. In particular, we show that the updated bound is tighter in all but only a relatively \enquote{small} number of cases using both formal techniques and via Monte Carlo simulation techniques. 

\vspace{1.0mm}

\vspace{1.0mm}
\end{abstract}


\section{Introduction}
Let $\boldsymbol{a}$ be a positive integer $n$-dimensional primitive vector, i.e. $\boldsymbol{a} = (a_1, \ldots, a_n)^T \in \mathbb{Z}^n_{>0}$ with $\gcd(\boldsymbol{a}) := \gcd(a_1, \ldots, a_n) =1$. In what follows, we exclude the case $n=1$ and assume that the dimension $n \ge 2$. In particular, without loss of generality, we assume the following conditions:
\begin{equation} \label{conditions on a}
\boldsymbol{a} = (a_1, \ldots, a_n)^T \in \mathbb{Z}^n_{>0}\,, n \ge 2 \text{ and } \gcd(\boldsymbol{a}):= \gcd(a_1, \ldots, a_n) = 1.
\end{equation}
The \textit{Frobenius number} of $\boldsymbol{a}$, denoted by $F(\boldsymbol{a})$, is the largest integer that cannot be represented as a nonnegative integer combination of the $a_i$'s, i.e.
$$
F(\boldsymbol{a}) := \max \left\{ b \in \mathbb{Z} : b \ne \boldsymbol{a}^T \boldsymbol{z} \text{ for all } \boldsymbol{z} \in \mathbb{Z}^n_{\ge 0} \right\} ,
$$
where $\boldsymbol{a}^T$ denotes the transpose of $\boldsymbol{a}$. It should be noted for completeness that the Frobenius problem, namely the problem of finding the Frobenius number, is also known by other names within the literature including the money-changing problem (or the money-changing problem of Frobenius, or the coin-exchange problem of Frobenius) \cite{wilf1978circle, tripathi2003variation, bocker2005money}, the coin problem (or the Frobenius coin problem) \cite{boju2007math, spivey2007quadratic} and the Diophantine problem of Frobenius \cite{selmer1977linear, rodseth1978linear}. From a geometric viewpoint, $F(\boldsymbol{a})$ is the maximal right-hand side $b \in \mathbb{Z}$ such that the \textit{knapsack polytope} 
$$
P(\boldsymbol{a},b) = \left\{ \boldsymbol{x} \in \mathbb{R}^n_{\ge 0} : \boldsymbol{a}^T \boldsymbol{x} = b
\right\}
$$
does not contain integer points. It should be noted that \eqref{conditions on a} is indeed a necessary and sufficient condition for the existence of the Frobenius number.


Note that instead of the conditions \eqref{conditions on a}, some authors instead assume the stronger condition that all the entries of the vector are pairwise coprime, i.e. 
\begin{equation} \label{stronger conditions on a}
\boldsymbol{a} = (a_1, \ldots, a_n)^T \in \mathbb{Z}^n_{>0}\,, n \ge 2 \text{ and } \gcd(a_i, a_j) = 1 \text{ for any } i,j \in \{1,2,\ldots, n\} \text{ with } i \ne j.
\end{equation}
It should be noted that not all integer vectors satisfying \eqref{conditions on a} also satisfy the stronger conditions \eqref{stronger conditions on a}. The vector $\boldsymbol{a} = (6,10,15)^T$ for example satisfies $\gcd(6,10,15) = 1$, but not the pairwise coprime condition, since $\gcd(6,10) \ne 1$.


There is a very rich history on Frobenius numbers and the book \cite{alfonsin2005diophantine} provides a very good survey of the problem. It is worth noting that computing the Frobenius number in general is $\mathcal{NP}$-hard \cite{ramirez1996complexity} (which was proved via a reduction to the integer knapsack problem), however, if the number of integers $n$ is fixed, then a polynomial time algorithm to calculate $F(\boldsymbol{a})$ exists \cite{kannan1992lattice}. If $n=2$, it is well-known (most likely due to Sylvester \cite{sylvester1884problem}) that 
\begin{equation} \label{Sylvester 2nd Frobenius bound}
\begin{aligned}
F(a_1, a_2) &= a_1 a_2 - a_1 - a_2 \\
&= (a_1 - 1) (a_2 - 1) - 1\,.
\end{aligned}
\end{equation}

In contrast to the case when $n=2$, it was shown by Curtis \cite{curtis1990formulas} that no closed formula exists for the Frobenius number if $n>2$. In light of this, there has been a great deal of research into producing upper bounds on $F(\boldsymbol{a})$. These bounds share the property that in the worst-case they are of quadratic order with respect to the maximum absolute valued entry of $\boldsymbol{a}$, which will be denoted by $\| \boldsymbol{a} \|_{\infty}$. Further, let $\| \cdot \|_{2}$ denote the Euclidean norm.
In particular, upon assuming that $a_1 \le a_2 \le \cdots \le a_n$ holds, such bounds include the classical bound by Erd{\H o}s and Graham \cite[Theorem 1]{erdos1972linear}
\begin{equation*}
F(\boldsymbol{a}) \le 2 a_{n-1} \left\lfloor \frac{a_n}{n} \right\rfloor - a_n ,
\end{equation*}
by Selmer \cite{selmer1977linear}
\begin{equation*}
F(\boldsymbol{a}) \le 2 a_n \left\lfloor \frac{a_1}{n} \right\rfloor - a_1 ,
\end{equation*}
by Vitek \cite[Theorem 5]{vitek1975bounds}
\begin{equation*}
F(\boldsymbol{a}) \le \frac{1}{2}(a_2 - 1) (a_n - 2) - 1,
\end{equation*}
by Beck et al. \cite[Theorem 9]{beck2002frobenius}
\begin{equation*}
F(\boldsymbol{a}) \le \frac{1}{2}\left(\sqrt{a_1 a_2 a_3\left(a_1+a_2+a_3\right)}-a_1-a_2-a_3\right),
\end{equation*}
and by Fukshansky and Robins \cite[Equation 29]{fukshansky2007frobenius}
\begin{equation*}
F(\boldsymbol{a}) \le \left\lfloor \frac{(n-1)^{2} \,\Gamma (\frac{n+1}{2})}
{\pi^{(n-1)/ 2}} \sum_{i=1}^{n} a_{i} \sqrt{\|\boldsymbol{a}\|_{2}^{2}-a_{i}^{2}}+1 \right\rfloor ,
\end{equation*}
where $\Gamma (\cdot)$ and $\lfloor \, \cdot \, \rfloor$ denote Euler's gamma and the standard floor functions, respectively. 

It is worth noting that providing accurate upper bounds in the general setting, namely without additional assumptions on the vector $\boldsymbol{a}$, is not the only direction of research. In particular, there has been results on lower bounds for $F(\boldsymbol{a})$ (e.g. \cite{davison1994linear, killingbergtro2000frobenius, aliev2007frobenius}),  some explicit formulas provided in special cases 
(e.g. \cite{roberts1957diophantine, brauer1962frobenius, selmer1977linear, tripathi2017frobenius, komatsu2022frobenius, komatsu2023frobenius, roblesperez2018frobenius, tripathi2016frobenius, liu2024frobenius}) 
and algorithms for computing the Frobenius number 
(e.g. \cite{wilf1978circle, kannan1992lattice, nijenhuis1979money,  greenberg1980frobenius, beihoffer2005frobenius, johnson1960diophantine, selmer1978frobenius, morales2024frobenius}). 

Building on these directions of research, this paper is motivated by the need for accurate and reliable upper bounds on the Frobenius number, particularly in settings with minimal assumptions on the input vector $\boldsymbol{a}$. Such bounds play a crucial role in understanding the complexity of the Frobenius problem and its connections to optimization problems such as the knapsack (see e.g. \cite[Chapter 16]{cormen2009introduction}) and subset-sum (see e.g. \cite[Chapter 35]{cormen2009introduction}) problems.

This paper primarily focuses on refining a well-known upper bound, originally proposed by Beck et al. \cite[Theorem 9]{beck2002frobenius}. While this bound has been cited in the literature, we identify a subtle error in its derivation that, while not invalidating the bound itself, affects its tightness. Our primary contribution is the formal correction of this result, alongside a rigorous comparison of the relative tightness of the corrected and original bounds. Through theoretical analysis and Monte Carlo simulations, we demonstrate that the corrected bound is tighter in ``{nearly all}'' cases. This work not only addresses a key issue in the existing literature but also enhances our understanding of the structure of upper bounds under general settings, paving the way for further advancements in this area of research.


\section{Preliminary and Auxiliary Results}
In this section, we present some preliminary results that are essential for establishing a requirement for upper bounds on $F(\boldsymbol{a})$ when the more general conditions \eqref{conditions on a} on $\boldsymbol{a}$ hold. In particular, this requirement is induced via a simple lower bound which considers the parity of the $a_i$'s and is formally introduced below. 

\begin{Proposition} \label{Proposition_ODD_Exists}
If an integer vector $\boldsymbol{a}$ satisfies \eqref{conditions on a}, then at least one $a_i$ must be odd for $i \in \{1,2,\ldots, n\}$.
\end{Proposition}

\begin{proof}
Suppose for contradiction there does not exist an odd $a_i$, i.e. that $\boldsymbol{a}$ has only even elements. It follows immediately that $\gcd(\boldsymbol{a}) \ge 2$, which contradicts the assumed conditions \eqref{conditions on a} as required. 
\end{proof}

Denote by $o_t := o_t (\boldsymbol{a})$ the $t$-th smallest odd element in $\boldsymbol{a}$. Observe that Proposition \ref{Proposition_ODD_Exists} implies that $o_1$ necessarily exists for any integer vector $\boldsymbol{a}$ satisfying \eqref{conditions on a}.

\begin{Proposition} \label{Proposition_Lower_Bound}
If an integer vector $\boldsymbol{a}$ satisfies \eqref{conditions on a}, then $o_1 - 2$ is a lower bound for $F(\boldsymbol{a})$.
\end{Proposition}

\begin{proof}
Firstly, observe that since $o_1$ is the smallest odd element in $\boldsymbol{a}$, it follows that any odd number strictly less than $o_1$ cannot be expressed as $\sum_{i=1}^n a_i x_i$ for $x_i \in \mathbb{Z}_{\ge 0}$ for $i \in \{1,2,\ldots,n\}$. In particular, $o_1 - 2$ cannot be expressed as a nonnegative integer linear combination of the $a_i$'s. Thus, the Frobenius number $F(\boldsymbol{a})$ is at least $o_1 - 2$ as required. 
\end{proof}

The propositions outlined can be applied to establish a requirement for upper bounds on the Frobenius number, particularly where the weaker conditions \eqref{conditions on a} concerning the vector $\boldsymbol{a}$ are met.

\begin{Lemma} \label{dependency_on_maximum}
If an integer vector $\boldsymbol{a} = (a_1, a_2, \ldots, a_n)^T$ satisfies \eqref{conditions on a} with $a_1 \le a_2 \le \cdots \le a_n$, then any general upper bound on the Frobenius number $F(\boldsymbol{a})$ must inherently depend on the largest element $a_n$. 
\end{Lemma}

\begin{proof}
Let us suppose for simplicity that the vector $\boldsymbol{a}$ has the form that $a_i$ is even for each $i \in \{1,2,\ldots,n-1\}$ while the final entry $a_n$ is odd. Notice that here $o_1 = a_n$ and, in light of Proposition \ref{Proposition_Lower_Bound}, it immediately follows that $F(\boldsymbol{a}) \ge a_n -2$. 

Suppose for contradiction that there exists an upper bound on $F(\boldsymbol{a})$ that does not depend on $a_n$, i.e. that there exists some function $f:\mathbb{R}^{n-1} \rightarrow \mathbb{R}$ satisfying $F(\boldsymbol{a}) \le f(a_1, a_2, \ldots, a_{n-1})$. If we set $a_n = f(a_1, a_2, \ldots, a_{n-1}) + 3$ if $f(a_1, a_2, \ldots, a_{n-1})$ is even and $a_n = f(a_1, a_2, \ldots, a_{n-1}) + 4$ if $f(a_1, a_2, \ldots, a_{n-1})$ is odd, then observe that $a_n -2 > f(a_1, a_2, \ldots, a_{n-1})$ holds. In particular, notice that the lower bound on the Frobenius number is strictly larger than the (assumed) upper bound, which is a contradiction as required. 
\end{proof}

It should be emphasised that this result suggests that any (general) upper bound on the Frobenius number which does not depend on the maximal entry of $\boldsymbol{a}$ does not necessarily hold in general without stronger assumptions than the conditions \eqref{conditions on a}.

The following results provide a rather surprisingly useful property that holds when $\boldsymbol{a}$ satisfies the stronger conditions \eqref{stronger conditions on a} that all the entries of the vector are pairwise coprime.

\begin{Lemma} \label{Upper_Bound_from_Sylvester}
If an integer vector $\boldsymbol{a} = (a_1, a_2, \ldots, a_n)^T$ satisfies \eqref{stronger conditions on a}, then for any $i,j \in \{1,2,\ldots,n\}$ with $i \ne j$ we have 
$$
F(\boldsymbol{a}) \le (a_i - 1)(a_j - 1) - 1.
$$
\end{Lemma}

\begin{proof}
Firstly, notice that given any pair $a_i$ and $a_j$ with $i \ne j$ that in light of the conditions \eqref{stronger conditions on a} it follows that $\gcd(a_i, a_j) = 1$. Thus, the Frobenius number $F(a_i, a_j)$ corresponding to the pair $a_i$ and $a_j$ exists and takes finite value. Furthermore, it follows in light of \eqref{Sylvester 2nd Frobenius bound} that the equality 
$$
F(a_i, a_j) = (a_i - 1) (a_j -1) - 1
$$ 
holds. By the definition of Frobenius number, we deduce that all integers strictly greater than $(a_i - 1) (a_j -1) - 1$ can be expressed as $a_i x_i + a_j x_j$ for some $x_i, x_j \in \mathbb{Z}_{\ge 0}$. Thus, it immediately follows that all integers strictly greater than $(a_i - 1) (a_j -1) - 1$ can be expressed as $\sum_{k=1}^n a_k x_k$ for $x_k \in \mathbb{Z}_{\ge 0}$ for $k \in \{1,2,\ldots,n\}$ (upon setting $x_k = 0$ when $k \ne i, j$ whenever necessary). In particular, this shows that the Frobenius number $F(\boldsymbol{a})$ satisfies the inequality $F(\boldsymbol{a}) \le (a_i - 1)(a_j - 1) - 1$ for any $i \ne j$ as required. 
\end{proof}

The following corollary follows immediately from Lemma \ref{Upper_Bound_from_Sylvester}. 

\begin{Corollary} \label{Improved_Upper_Bound_from_Sylvester}
If an integer vector $\boldsymbol{a} = (a_1, a_2, \ldots, a_n)^T$ satisfies \eqref{stronger conditions on a} and $a_1 \le a_2 \le \cdots \le a_n$, then 
$$
F(\boldsymbol{a}) \le (a_1 - 1)(a_2 - 1) - 1.
$$
\end{Corollary}

It should be emphasised that the above results tell us that the well-known result \eqref{Sylvester 2nd Frobenius bound} of Sylvester \cite{sylvester1884problem} extends naturally to provide an upper bound for the Frobenius number $F(\boldsymbol{a})$ under the (stronger) assumption \eqref{stronger conditions on a} that the entries of the vector $\boldsymbol{a}$ are pairwise coprime.


\section{Observations on a Previously Known Upper Bound}
Recall that Beck et al. \cite[Theorem 9]{beck2002frobenius} introduced the upper bound 
\begin{equation} \label{Beck Original Bound}
F(\boldsymbol{a}) \le \frac{1}{2}\left(\sqrt{a_1 a_2 a_3\left(a_1+a_2+a_3\right)}-a_1-a_2-a_3\right)
\end{equation}
on the Frobenius number upon finding bounds for Fourier-Dedekind sums. This bound \eqref{Beck Original Bound} is widely referenced across books and papers, however, in most of these little attention is given to the underlying assumed conditions on $\boldsymbol{a}$. In particular, the upper bound \eqref{Beck Original Bound} necessitates that the stronger conditions \eqref{stronger conditions on a} hold, instead of the more general (weaker) conditions \eqref{conditions on a}.

\begin{Proposition}
The upper bound \eqref{Beck Original Bound} of Beck et al. \cite[Theorem 9]{beck2002frobenius} does not necessarily hold unless the stronger conditions \eqref{stronger conditions on a} hold. This requirement remains even if the weaker conditions \eqref{conditions on a} are met.
\end{Proposition}

\begin{proof}
Observe that if $n=2$, then the stronger \eqref{stronger conditions on a} and weaker conditions \eqref{conditions on a} are equivalent. Thus, we focus here only on the case that $n \ge 3$, where we show there are counterexamples for each $n$.

Let us consider two cases, namely $n=3$ and $n \ge 4$, respectively. If $n=3$, then consider the integer vector $\boldsymbol{a} = (3,6,19)^T$ with $F(\boldsymbol{a}) = 35$. In such case, notice that the bound \eqref{Beck Original Bound} yields
\begin{equation} \label{counter}
\frac{1}{2} \left( \sqrt{3 \cdot 6 \cdot 19 \, (3 + 6 + 19)} - 3 - 6 - 19 \right) = \frac{1}{2} \left( 6 \sqrt{266} - 28 \right) \approx 34.928519
\end{equation}
where, in particular, $35 \not\le 34.928519$ and hence the upper bound \eqref{Beck Original Bound} fails when $n = 3$. If instead $n \ge 4$, then in light of Lemma \ref{dependency_on_maximum} clearly \eqref{Beck Original Bound} cannot be a general upper bound for the Frobenius number.
\end{proof}

Note that in the case $n \ge 4$, any vector $\boldsymbol{a} = (2,4,6, a_4,\ldots,a_n)^T$ satisfying \eqref{conditions on a} and $a_n \ge \cdots \ge a_4 > 7$ provides a counterexample for any $n$. Indeed, since $a_4 > 7$ by assumption, the Frobenius number is clearly greater than or equal to 7 (since 7 cannot be expressed by $\sum_{i=1}^n a_i x_i$ for $x_i \in \mathbb{Z}_{\ge 0}$ for all $i$). Despite this, the bound \eqref{Beck Original Bound} yields 
\begin{equation} \label{counter2}
\frac{1}{2} \left( \sqrt{2 \cdot 4 \cdot 6 \, (2 + 4 + 6)} - 2 - 4 - 6 \right) = 6, 
\end{equation}
which demonstrates the upper bound \eqref{Beck Original Bound} does not necessarily hold if only the weaker conditions  \eqref{conditions on a} are met.

The remarks presented in this section are intended to clarify a common misunderstanding about the upper bound \eqref{Beck Original Bound} as referenced in various books and papers. Furthermore, it is crucial to highlight a subtle error in the argument presented by Beck et al. \cite{beck2002frobenius}, which alters the value of the upper bound. The following result states the corrected upper bound, where the proof is outlined in a later section.

\begin{Theorem} \label{Corrected_Beck_bound}
If an integer vector $\boldsymbol{a} = (a_1, a_2, \ldots, a_n)^T$ satisfies \eqref{stronger conditions on a} with $a_1 \le a_2 \le \cdots \le a_n$, then the argument of Beck et al. \cite{beck2002frobenius} yields
\begin{equation} \label{updated_Beck_bound_eq}
\begin{aligned}
F(\boldsymbol{a}) &\le \frac{1}{2}\bigg(\sqrt{\frac{1}{3}\left(a_1+a_2+a_3\right)\left(a_1+a_2+a_3+2 a_1 a_2 a_3\right)+\frac{8}{3}\left(a_1 a_2+a_2 a_3+a_3 a_1\right)} \\
&\hspace{9.5cm} -a_1-a_2-a_3\bigg).
\end{aligned}
\end{equation}
\end{Theorem}

It is natural to consider if the original bound \eqref{Beck Original Bound} given by Beck et al. \cite[Theorem 9]{beck2002frobenius} is indeed correct provided the integer vector $\boldsymbol{a}$ satisfies the stronger conditions \eqref{stronger conditions on a}. It turns out that the upper bound \eqref{Beck Original Bound} remains valid. The following result states this formally, where the proof is outlined in a later section. 

\begin{Theorem} \label{Original_Corrected_Beck_bound} 
If an integer vector $\boldsymbol{a} = (a_1, a_2, \ldots, a_n)^T$ satisfies \eqref{stronger conditions on a} with $a_1 \le a_2 \le \cdots \le a_n$, then 
\begin{equation*}
F(\boldsymbol{a}) \le \frac{1}{2}\left(\sqrt{a_1 a_2 a_3\left(a_1+a_2+a_3\right)}-a_1-a_2-a_3\right).
\end{equation*}
\end{Theorem}

Furthermore, it is natural to consider the relative tightness of \eqref{Corrected_Beck_bound} with \eqref{Beck Original Bound}. This comparison will be explored in the subsequent section of the paper.


\section{Tightness Comparison of Upper Bounds}
In this section, we consider the relative tightness of the upper bounds \eqref{Corrected_Beck_bound} and \eqref{Beck Original Bound}. In particular, to slightly simplify notation, let us denote by 
$$
\begin{aligned}
UB_1(\boldsymbol{a}) &= UB_1 (a_1, a_2, a_3) \\
&:= \frac{1}{2} \bigg( 
    \sqrt{ 
        \frac{1}{3}\big(a_1 + a_2 + a_3\big) 
        \big(a_1 + a_2 + a_3 + 2a_1a_2a_3\big) 
        + \frac{8}{3}\big(a_1a_2 + a_2a_3 + a_3a_1\big)
    } \\
&\hspace{10cm} - a_1 - a_2 - a_3 
\bigg)
\end{aligned}
$$
and
$$
UB_2 (\boldsymbol{a}) = UB_2 (a_1, a_2, a_3) :=   \frac{1}{2}\left(\sqrt{a_1 a_2 a_3\left(a_1+a_2+a_3\right)}-a_1-a_2-a_3\right).
$$
Recall that Theorem \ref{Corrected_Beck_bound} and \ref{Original_Corrected_Beck_bound} imply that $UB_1(\boldsymbol{a})$ and $UB_2(\boldsymbol{a})$ are valid upper bounds provided $\boldsymbol{a}$ satisfies \eqref{stronger conditions on a} and $a_1 \le a_2 \le \cdots \le a_n$, however, we instead are interested in which bound is tighter. The first result of this section shows that $UB_2 (\boldsymbol{a})$ is tighter than $UB_1 (\boldsymbol{a})$ is only a relatively ``{small}'' (finite) number of cases, where the proof is excluded given this was completed via simple enumeration.

\begin{Theorem} \label{UB2_vs_UB1_Finite_Cases}
If an integer vector $\boldsymbol{a} = (a_1, a_2, \ldots, a_n)^T$ satisfies \eqref{stronger conditions on a} with $a_1 \le a_2 \le \cdots \le a_n$, then $F(\boldsymbol{a})$ satisfies $F(\boldsymbol{a}) \le UB_1 (\boldsymbol{a})$ and $F(\boldsymbol{a}) \le UB_2 (\boldsymbol{a})$, where $UB_2(\boldsymbol{a})$ is sharper only when
$$
\begin{aligned}
(a_1, a_2, a_3) \in \big\{  &(1,2,3),(1,2,5),(1,2,7),(1,2,9),(1,2,11),(1,2,13),(1,2,15), \\
    &(1,2,17), (1,2,19), (1,2,21),(1,2,23),(1,2,25),(1,3,4),(1,3,5), \\
    &(1,3,7),(1,3,8), (1,3,10), (1,3,11),(1,3,13),(1,3,14),(1,4,5), \\
    &(1,4,7),(1,4,9),(1,4,11), (1,5,6), (1,5,7),(1,5,8),(1,5,9),(1,6,7),(2,3,5) \big\} .
\end{aligned}
$$
\end{Theorem}

It should be emphasised that this result suggests that in ``{almost all}'' cases, $UB_1(\boldsymbol{a})$ provides a tighter bound than $UB_2(\boldsymbol{a})$. Note that in the above we assume the stronger conditions \eqref{stronger conditions on a} rather than the weaker conditions \eqref{conditions on a} since the examples \eqref{counter} and \eqref{counter2} show that $UB_1(\boldsymbol{a})$ and $UB_2(\boldsymbol{a})$ do not necessarily apply under only \eqref{conditions on a}. In order to compare the relative tightness of the bounds, for completeness, we apply Monte Carlo simulation techniques (see e.g. \cite[Chapter 2]{ross2012simulation}) and present the results below. 

During this simulation, we firstly randomly generate integer vectors $\boldsymbol{a}$ satisfying the conditions \eqref{stronger conditions on a} with ordering $a_1 \le a_2 \le \cdots \le a_n$, before computing the values of $UB_1(\boldsymbol{a})$ and $UB_2(\boldsymbol{a})$. This process is iteratively repeated 100,000 times. During the sampling, we set $\| \boldsymbol{a} \|_{\infty} = \max_i | a_i | \le 1000$ for convenience. Note that in each graph in Figure \ref{fig:mainfigureSimulation}, the vertical axis corresponds to the difference $UB_1(\boldsymbol{a}) - UB_2(\boldsymbol{a})$, where a large vertical value illustrates that the corrected upper bound \eqref{updated_Beck_bound_eq} is much tighter than the originally stated bound \eqref{Beck Original Bound}.

\vspace{2.0mm}

\begin{figure}[ht]
\centering
    \begin{subfigure}{0.49\linewidth}
        \includegraphics[width=\linewidth]{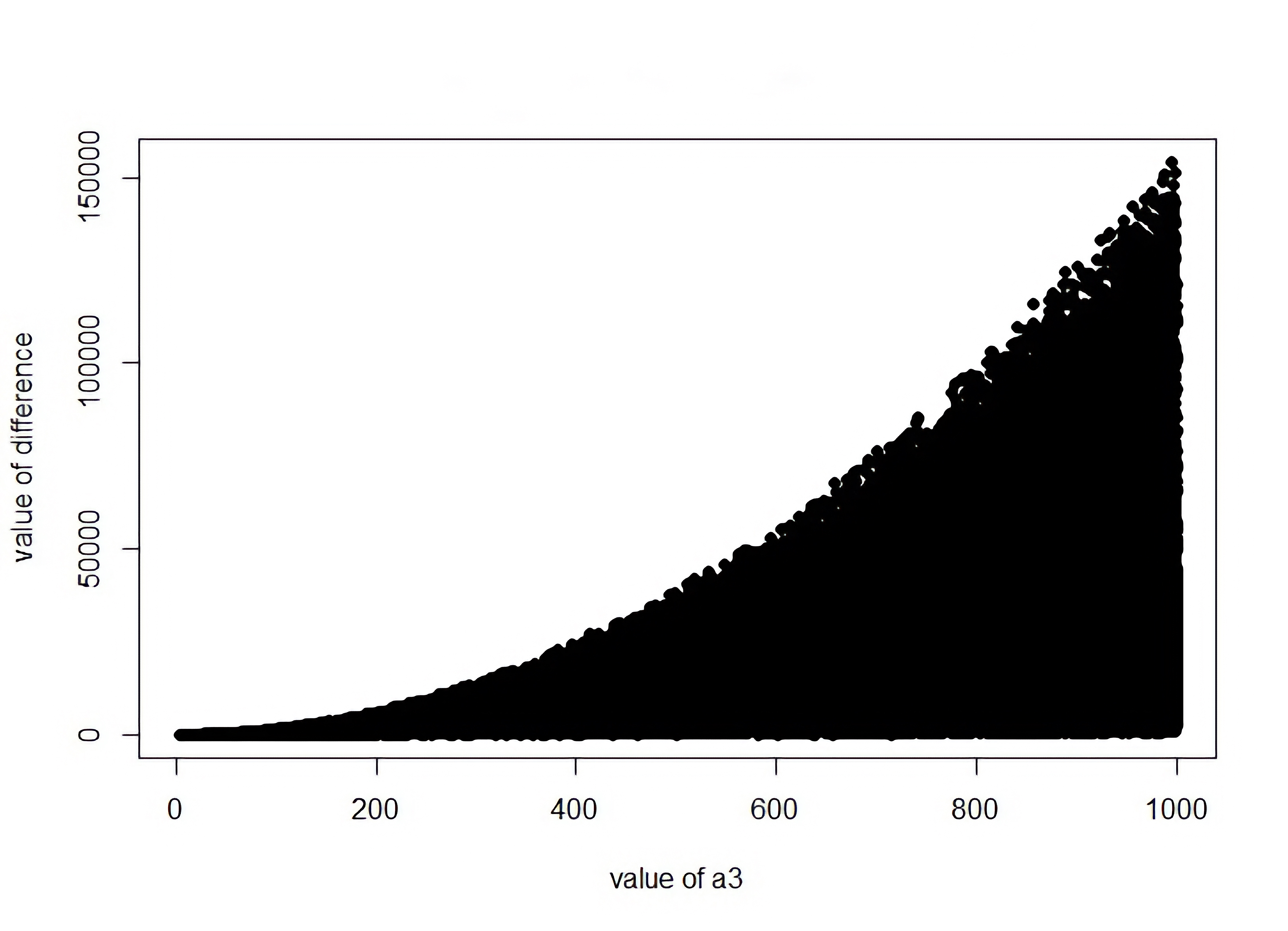}
        \caption{$UB_1(\boldsymbol{a}) - UB_2(\boldsymbol{a})$ upon increasing \\ the entry $a_3$.}
        \label{fig:Figure1}
    \end{subfigure}
    \hfill
    \begin{subfigure}{0.46\linewidth}
        \includegraphics[width=\linewidth]{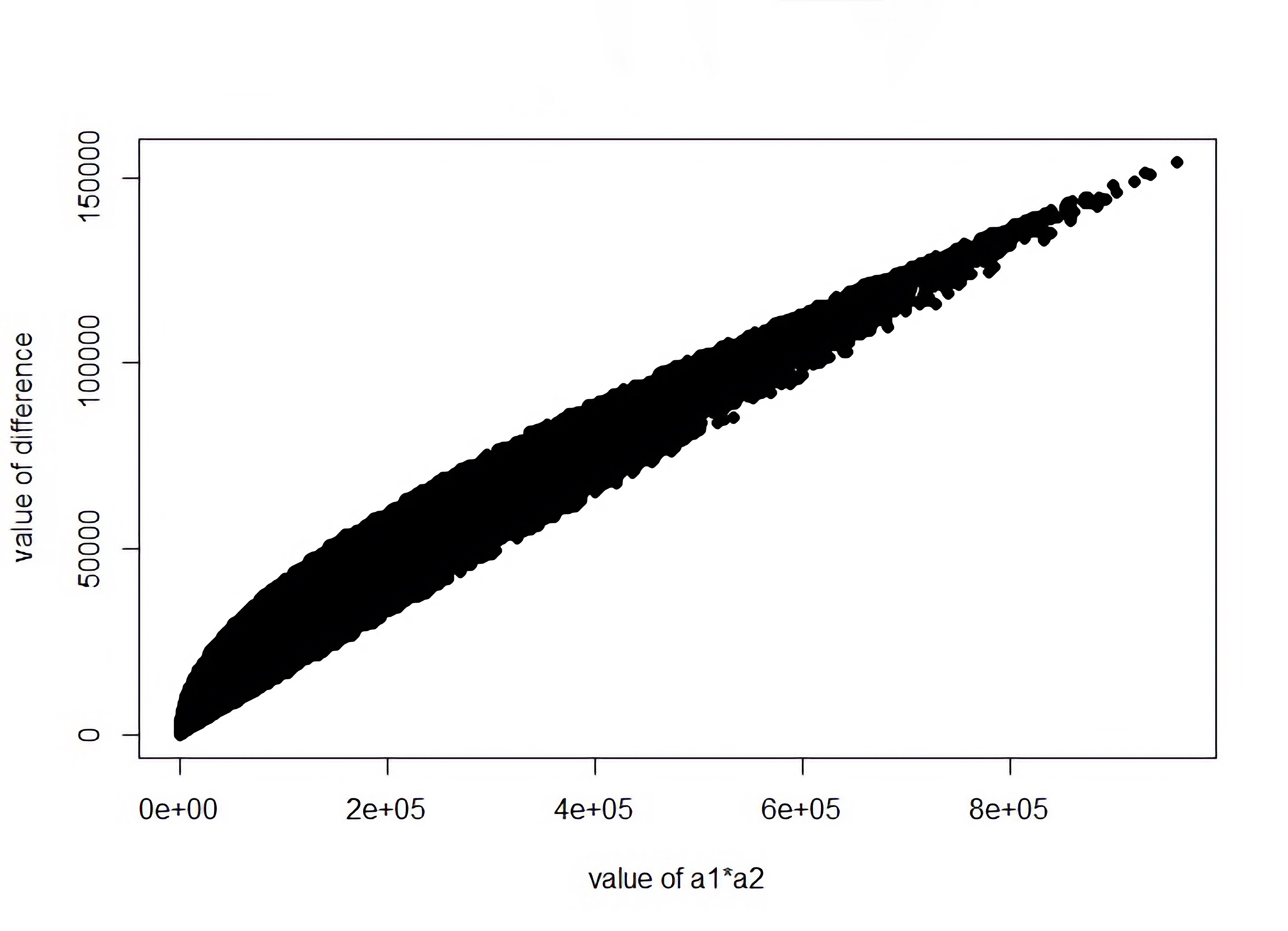}
        \caption{$UB_1(\boldsymbol{a}) - UB_2(\boldsymbol{a})$ upon increasing \\ the product $a_1a_2$.}
        \label{fig:Figure2}
    \end{subfigure}

\vspace{0.5em} 

    \begin{subfigure}{0.49\linewidth}
        \includegraphics[width=\linewidth]{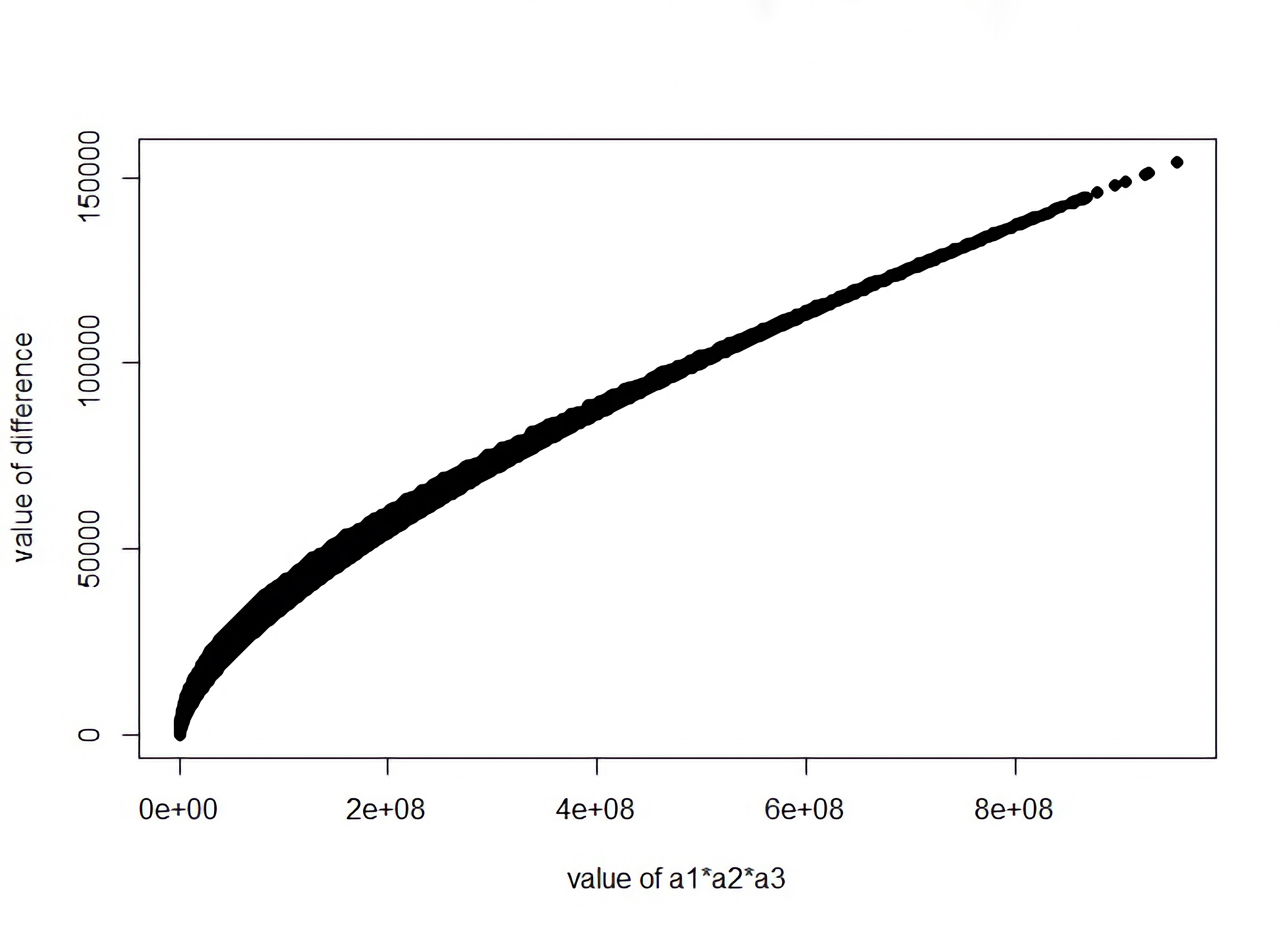}
        \caption{$UB_1(\boldsymbol{a}) - UB_2(\boldsymbol{a})$ upon increasing \\ the product $a_1 a_2 a_3$.}
        \label{fig:Figure3}
    \end{subfigure}
    \hfill
    \begin{subfigure}{0.49\linewidth}
        \includegraphics[width=\linewidth]{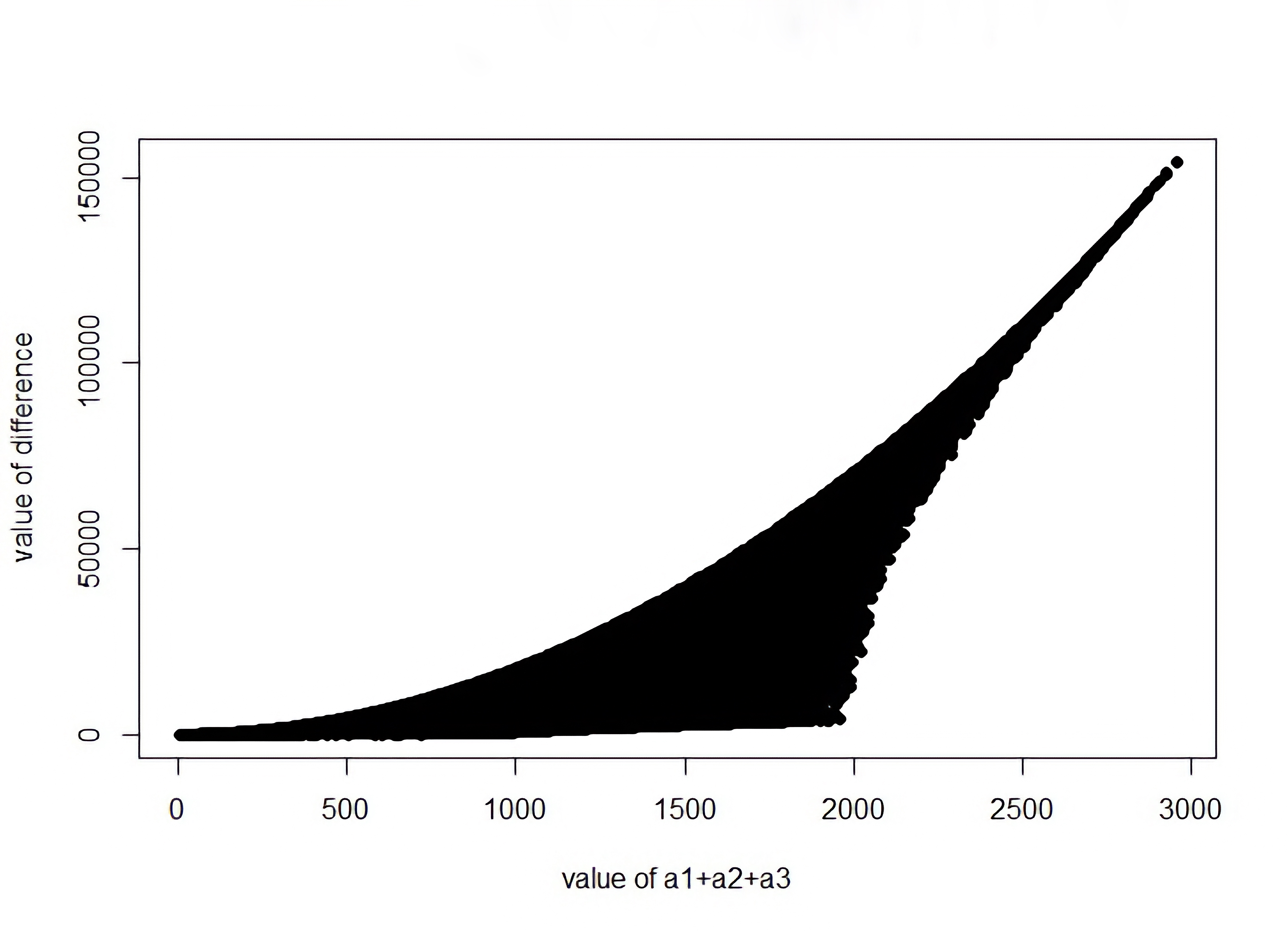}
        \caption{$UB_1(\boldsymbol{a}) - UB_2(\boldsymbol{a})$ upon increasing \\ the sum $a_1+a_2+a_3$.}
        \label{fig:Figure4}
    \end{subfigure}
    \caption{This figure illustrates how the difference $UB_1(\boldsymbol{a}) - UB_2(\boldsymbol{a})$ varies dependent upon the $a_i$'s.}
    \label{fig:mainfigureSimulation}
\end{figure}

Figure \ref{fig:Figure1} demonstrates the difference $UB_1(\boldsymbol{a}) - UB_2(\boldsymbol{a})$ grows rapidly with increases in $a_3$, while the difference remains small if $a_3$ is small. It should be emphasised that one would not expect the difference to be significant when $a_3$ is small given that the entries of $\boldsymbol{a}$ are ordered by assumption. Figure \ref{fig:Figure2} shows how this differences varies upon increases in $a_1a_2$. Notably, the figure suggests that the minimum difference increases linearly with $a_1a_2$, whereas the maximum difference seems to grow sublinearly with $a_1a_2$. Furthermore, if $a_1 a_2$ is large (around 1,000,000), then the variance of the difference $UB_1(\boldsymbol{a}) - UB_2(\boldsymbol{a})$ appears small. This can be explained by considering the difference
$$
\begin{aligned}
    UB_1(\boldsymbol{a}) - UB_2(\boldsymbol{a}) =\frac{1}{2}\Big(&\sqrt{a_1 a_2 a_3\left(a_1+a_2+a_3\right)} \\
    & -\sqrt{\frac{1}{3} a_1 a_2 a_3\left(a_1+a_2+a_3\right)+\left(a_1+a_2+a_3\right)^2+\frac{8}{3}\left(a_1 a_2+a_2 a_3+a_3 a_1\right)} \, \Big) ,
\end{aligned}
$$
which is maximised for fixed $a_1 a_2$ when $a_3$ is large and is minimised when $a_3$ is small. In particular, note that if $a_3$ is small, then the assumed ordering implies that $a_3 \approx a_2$. In such case, the value of the difference can be well approximated by 
$$
\frac{1}{2} a_1 a_2\left(\sqrt{3}-\sqrt{2+\frac{11}{a_1 a_2}} \, \right) ,
$$
which grows roughly linearly. If instead $a_1 a_2$ is large (say around 1,000,000), then the assumed ordering and upper bound on the $a_i$'s restricts variance in both $UB_1(\boldsymbol{a})$ and $UB_2(\boldsymbol{a})$, respectively. Figure \ref{fig:Figure3} demonstrates the difference $UB_1(\boldsymbol{a}) - UB_2(\boldsymbol{a})$ grows sublinearly with the product $a_1 a_2 a_3$, while the variance once more appears small, which can be similarly explained via careful algebraic analysis. Figure \ref{fig:Figure4} shows how this difference varies upon increases in $a_1 + a_2 + a_3$, where the variance decreases significantly as the value of $a_1 + a_2 + a_3$ grows beyond 2000. Notably, the final figure suggests that one should expect the difference $UB_1(\boldsymbol{a}) - UB_2(\boldsymbol{a})$ to be small only in the scenario that $a_1$ is ``{reasonably small}'', which follows in light of the assumed ordering.


\section{Proof of Theorem \ref{Corrected_Beck_bound}}
In this section, we follow closely the argument presented by Beck et al. \cite[Theorem 9]{beck2002frobenius} to demonstrate that it actually yields the upper bound  \eqref{updated_Beck_bound_eq} instead of \eqref{Beck Original Bound}. It should be noted that the proof presented by Beck et al. \cite{beck2002frobenius} instead provides an upper bound for
$$
\begin{aligned}
F^*(\boldsymbol{a}) :&= \max \left\{ b \in \mathbb{Z} : b \ne \boldsymbol{a}^T \boldsymbol{z} \text{ for all } \boldsymbol{z} \in \mathbb{Z}^n_{> 0} \right\} \\
&= F(\boldsymbol{a}) + a_1 + a_2 + \cdots + a_n, 
\end{aligned}
$$
which is the largest integer that cannot be represented as a (strictly) positive integer combination of the $a_i$'s. 

Let $A = \{ a_1, a_2, \ldots, a_n \}$ be a set of pairwise coprime positive integers, and define the function 
$$
p_A^{\prime}(b)=\#\left\{\left(m_1, \ldots, m_n\right)^T \in \mathbb{Z}^n_{>0}: \sum_{k=1}^n m_k \, a_k=b\right\}, 
$$
where $\#$ denotes the cardinality of the set. Specifically, $p_A^{\prime}(b)$ counts the number of (strictly) positive tuples $\left(m_1, \ldots, m_n\right)^T \in \mathbb{Z}^n_{>0}$ satisfying the equality $\sum_{k=1}^n m_k \, a_k=b$. Notice that $F^*(\boldsymbol{a})$ is simply the largest value for $b$ for which $p_A^{\prime}(b) = 0$.

Let $c_1, c_2, \ldots, c_n \in \mathbb{Z}$ be relatively prime to $c \in \mathbb{Z}$, and $t \in \mathbb{Z}$. We define the Fourier–Dedekind sum as 
$$
\sigma_t\left(c_1, \ldots, c_n ; c\right)=\frac{1}{c} \sum_{\lambda^c=1 \neq \lambda} \frac{\lambda^t}{\left(\lambda^{c_1}-1\right) \cdots\left(\lambda^{c_n}-1\right)}.
$$
Note that one particularly noteworthy expression (which follows by periodicity) \cite{beck2002frobenius} is
\begin{equation} \label{useful_result}
\sigma_t(a, b ; c)=\sum_{m=0}^{c-1}\left(\left(\frac{-a^{-1}(b m+t)}{c}\right)\right)\left(\left(\frac{m}{c}\right)\right)-\frac{1}{4 c} 
\end{equation}
with $a a^{-1} \equiv 1 \pmod c$ and where $((x))=x-\lfloor x\rfloor-\frac{1}{2}$ is a sawtooth function.

\begin{proof}
Firstly, note that it is easy to verify that
$$
F^*(\boldsymbol{a}) = F^*(a_1, a_2, \ldots, a_n) \le F^*(a_1, a_2, a_3) + a_3 + a_4 + \cdots + a_n.
$$
We closely follow \cite{beck2002frobenius} by focusing on the case where $n=3$ and the $a_i$'s are pairwise coprime. In order to slightly simplify notation, let $a,b,c$ denote pairwise relatively prime positive integers. Upon using the the Cauchy-Schwartz inequality we find 
$$
\begin{aligned}
\sigma_t(a, b ; c) & \geqslant-\sum_{m=0}^{c-1}\left(\left(\frac{m}{c}\right)\right)^2-\frac{1}{4 c}=\sum_{m=0}^{c-1}\left(\frac{m}{c}-\frac{1}{2}\right)^2-\frac{1}{4 c} \\
& =-\frac{1(2 c-1)(c-1) c}{c^2}+\frac{1}{c} \frac{c(c-1)}{2}-\frac{c}{4}-\frac{1}{4 c} \\
& =-\frac{c}{12}-\frac{5}{12 c} \, .
\end{aligned}
$$
It should be noted that in \cite{beck2002frobenius}, the right-hand side of the expression previously discussed differs from the one presented here.
We can now utilise the above inequality to obtain 
$$
\begin{aligned}
p_{\{a, b, c\}}^{\prime}(t) \geq & \, \frac{t^2}{2 a b c}-\frac{t}{2}\left(\frac{1}{a b}+\frac{1}{a c}+\frac{1}{b c}\right)+\frac{1}{12}\left(\frac{3}{a}+\frac{3}{b}+\frac{3}{c}+\frac{a}{b c}+\frac{b}{a c}+\frac{c}{a b}\right) \\ 
& \, \hspace{7.0mm} -\frac{1}{12}(a+b+c)-\frac{5}{12}\left(\frac{1}{a}+\frac{1}{b}+\frac{1}{c}\right) \\ 
= & \, \frac{t^2}{2 a b c}-\frac{t}{2}\left(\frac{1}{a b}+\frac{1}{a c}+\frac{1}{b c}\right)+\frac{1}{12}\left(\frac{a}{b c}+\frac{b}{a c}+\frac{c}{a b}\right) \\
& \, \hspace{7.0mm} -\frac{1}{12}(a+b+c)-\frac{1}{6}\left(\frac{1}{a}+\frac{1}{b}+\frac{1}{c}\right) \, , 
\end{aligned}
$$
which upon algebraic manipulation yields the upper bound 
$$
F^*(a,b,c) \le \frac{1}{2}(a+b+c)+\frac{1}{2} \sqrt{\frac{1}{3}(a+b+c)(a+b+c+2 a b c)+\frac{8}{3}(a b+b c+c a)}
$$
Thus, upon replacing $a,b$ and $c$ with $a_1, a_2$ and $a_3$, respectively, we deduce that 
$$
\begin{aligned}
    F^*(\boldsymbol{a}) &\le F^*(a_1, a_2, a_3) + a_3 + a_4 + \cdots + a_n \\ 
    &\le \left( \frac{1}{2}\sqrt{\frac{1}{3}\left(a_1+a_2+a_3\right)\left(a_1+a_2+a_3+2 a_1 a_2 a_3\right)+\frac{8}{3}\left(a_1 a_2+a_2 a_3+a_3 a_1\right)} + a_1 + a_2 + a_3 \right) \\ 
    & \, \hspace{20.0mm} + a_3 + a_4 + \cdots + a_n \, 
\end{aligned}
$$
which yields that 
$$
\begin{aligned}
F(\boldsymbol{a}) &= F^*(\boldsymbol{a}) - a_1 - a_2 - \cdots - a_n \\
&\le \frac{1}{2}\bigg(\sqrt{\frac{1}{3}\left(a_1+a_2+a_3\right)\left(a_1+a_2+a_3+2 a_1 a_2 a_3\right)+\frac{8}{3}\left(a_1 a_2+a_2 a_3+a_3 a_1\right)} \\ 
&\hspace{10cm}-a_1-a_2-a_3\bigg)
\end{aligned}
$$
as required, which concludes the proof.
\end{proof}


\section{Proof of Theorem \ref{Original_Corrected_Beck_bound}}
Recall the notation 
$$
\begin{aligned}
UB_1(\boldsymbol{a}) 
&:= \frac{1}{2} \bigg( 
    \sqrt{ 
        \frac{1}{3}\big(a_1 + a_2 + a_3\big) 
        \big(a_1 + a_2 + a_3 + 2a_1a_2a_3\big) 
        + \frac{8}{3}\big(a_1a_2 + a_2a_3 + a_3a_1\big)
    } \\
&\hspace{10cm} - a_1 - a_2 - a_3 
\bigg)
\end{aligned}
$$
and
$$
UB_2 (\boldsymbol{a}) :=  \frac{1}{2}\left(\sqrt{a_1 a_2 a_3\left(a_1+a_2+a_3\right)}-a_1-a_2-a_3\right).
$$

In this section, we show that $UB_2 (\boldsymbol{a})$ is a valid upper bound on the Frobenius number $F(\boldsymbol{a})$ provided $\boldsymbol{a}$ satisfies \eqref{stronger conditions on a} and $a_1 \le a_2 \le \cdots \le a_n$. It should be noted that the stronger condition \eqref{stronger conditions on a} with the (assumed) ordering of the $a_i$'s implies that $a_1 < a_2 < \cdots < a_n$ provided that $a_1 > 1$. Before discussing the correctness of the upper bound $UB_2(\boldsymbol{a})$, we firstly consider how this bound varies with changes in $a_3$.

\begin{Proposition} \label{increasing_upper_bound}
If $a_1 < a_2 < a_3$, then $UB_2(\boldsymbol{a})$ is a strictly increasing function in $a_3$. If instead $a_1 \le a_2 \le a_3$ holds, then $UB_2(\boldsymbol{a})$ is a nondecreasing function in $a_3$.
\end{Proposition}

\begin{proof}
Upon partial differentiation with respect to $a_3$, observe that $U B_2(\boldsymbol{a})$ becomes
$$
\frac{\partial U B_2(\boldsymbol{a})}{\partial a_3}=\frac{1}{2}\left(\frac{a_1 a_2\left(a_1+a_2+2 a_3\right)}{2 \sqrt{a_1 a_2 a_3\left(a_1+a_2+a_3\right)}}-1\right).
$$
If $a_1 < a_2 < a_3$, then note that $a_1 a_2\left(a_1+a_2+2 a_3\right)>4 a_3$ holds. Upon simple algebraic manipulation, we deduce that $a_1^2 a_2^2\left(a_1+a_2+2 a_3\right)^2>4 a_1 a_2 a_3\left(a_1+a_2+a_3\right)$, which implies that 
$$
\frac{a_1 a_2\left(a_1+a_2+2 a_3\right)}{2 \sqrt{a_1 a_2 a_3\left(a_1+a_2+a_3\right)}}>1 .
$$
It follows that $\frac{\partial U B_2(\boldsymbol{a})}{\partial a_3} > 0$ and hence $UB_2(\boldsymbol{a})$ is a strictly increasing function in $a_3$ when $a_1 < a_2 < a_3$. If instead $a_1 \le a_2 \le a_3$, then a similar argument yields that $\frac{\partial U B_2(\boldsymbol{a})}{\partial a_3} \ge 0$ and hence $UB_2(\boldsymbol{a})$ is a nondecreasing function in $a_3$ as required. 
\end{proof}

We now proceed to present a detailed proof to establish the validity of Theorem \ref{Original_Corrected_Beck_bound}. 

\begin{proof}
Observe that if $UB_2 (\boldsymbol{a}) \ge UB_1 (\boldsymbol{a})$, then clearly the upper bound $UB_2 (\boldsymbol{a})$ is valid in consequence to the validity of Theorem \ref{Corrected_Beck_bound}. Thus, we consider only the setting where $UB_2 (\boldsymbol{a}) < UB_1 (\boldsymbol{a})$. Notice that upon simple algebraic manipulation the inequality $UB_2 (\boldsymbol{a}) \le UB_1 (\boldsymbol{a})$ is equivalent to 
\begin{equation} \label{UB_le_UB}
a_1 a_2 a_3\left(a_1+a_2+a_3\right) \leq a_1^2+a_2^2+a_3^2+10\left(a_1 a_2+a_2 a_3+a_1 a_3\right). 
\end{equation}

It is sufficient to here consider only the case that $a_1 < a_2 < a_3$. In particular, this follows because otherwise we require $a_1 = 1$ in light of the assumed conditions \eqref{stronger conditions on a} and hence, in such case, we yield that $F(\boldsymbol{a}) = -1 \le UB_2 (\boldsymbol{a})$ holds as required. 

In order to satisfy \eqref{UB_le_UB}, the (strict) inequality $a_1 a_2 < 33$ is necessary. Indeed, if instead $a_1 a_2 \ge 33$, then 
$$
\begin{aligned}
a_1 a_2 a_3\left(a_1+a_2+a_3\right) & =a_1 a_2 a_3^2+a_1 a_2 a_3\left(a_1+a_2\right) \\
& >a_1 a_2 a_3^2 \\
& \geq 33 a_3^2 \\
& = 3 a_3^2 + 10(a_3^2 + a_3^2 + a_3^2) \\
& >a_1^2+a_2^2+a_3^2+10\left(a_1 a_2+a_2 a_3+a_1 a_3\right)
\end{aligned}
$$
holds, where the final inequality follows since $a_1,a_2 < a_3$. Thus, we need only to consider the cases that $a_1 a_2 \le 32$ and $a_1 < a_2$ with $\gcd(a_1, a_2) = 1$. The pairs $(a_1, a_2)$ satisfying these conditions are:
\begin{enumerate}
    \item if $a_1 = 1$, then $\left(a_1, a_2\right)=(1,2),(1,3), \ldots,(1,31),(1,32)$, 
    \item if $a_1 = 2$, then $\left(a_1, a_2\right)=(2,3),(2,5),(2,7),(2,9),(2,11),(2,13),(2,15)$, 
    \item if $a_1 = 3$, then $\left(a_1, a_2\right)=(3,4),(3,5),(3,7),(3,8),(3,10)$, 
    \item if $a_1 = 4$, then $\left(a_1, a_2\right)=(4,5),(4,7)$, and
    \item if $a_1 = 5$, then $\left(a_1, a_2\right)=(5,6)$.
\end{enumerate}
It should be emphasised that if $UB_2 (\boldsymbol{a})$ is a valid upper bound in each of the above cases, then it follows that $UB_2 (\boldsymbol{a})$ is a valid upper bound as required. In order to complete the proof we now consider each of these cases in turn. It should be noted that cases (ii) - (v) use the properties $F(\boldsymbol{a}) \le (a_1 - 1)(a_2 - 1) - 1$ (which follows by Corollary \ref{Improved_Upper_Bound_from_Sylvester}) and $UB(a_1, a_2, a_3') > UB(a_1, a_2, a_3)$ when $a_3' > a_3$ (which follows by Proposition \ref{increasing_upper_bound}).

\begin{enumerate}
    \item $a_1 = 1$: In this case, notice that since the entries of the vector $\boldsymbol{a}$ are coprime by assumption, it follows that we have $a_2 \ge 2$ and $a_3 \ge 3$. Note that we have $\left(a_2-1\right)\left(a_3-1\right) \geq 2$ and $a_2 a_3 \geq 1+a_2+a_3$. These inequalities imply that
    $$
    a_2 a_3\left(1+a_2+a_3\right) \geq\left(1+a_2+a_3\right)^2
    $$
    which, upon rearranging algebraically, yields that 
    $$
    \frac{1}{2}\left(\sqrt{a_2 a_3\left(1+a_2+a_3\right)}-\left(1+a_2+a_3\right)\right) = UB_2 (1, a_2, a_3, \ldots, a_n) \geq 0. 
    $$
    Finally, observe that the equality $F(1, a_2, a_3, \ldots, a_n) = -1$ holds for all $a_2, a_3, \ldots, a_n$ and, thus, $UB_2(\boldsymbol{a})$ is a valid upper bound in this scenario.
    
    \item $a_1 = 2$:
    In this case, notice that 
    $$
    F(2,3,a_3,a_4,\ldots, a_n) \le (2-1) (3-1) - 1 = 1 < 3.660254 = UB_2(2,3,5) \le UB_2(2,3,5), 
    $$
    where the strict inequality follows since if $(a_1, a_2) = (2,3)$, then $a_3 \ge 5$ by the conditions \eqref{stronger conditions on a}. In a similar fashion, notice that 
    $$
    F(2,5,a_3,a_4,\ldots, a_n) \le (2-1)(5-1) -1 = 3 < 8.652476 = UB_2 (2,5,7) \le UB_2 (2,5,a_3),
    $$
    $$
    F(2,7,a_3,a_4,\ldots, a_n) \le (2-1)(7-1) -1 = 5 < 14.811762 = UB_2 (2,7,9) \le UB_2 (2,7,a_3),
    $$
    $$
    F(2,9,a_3,a_4,\ldots, a_n) \le (2-1)(9-1) -1 = 7 < 22 = UB_2 (2,9,11) \le UB_2 (2,9,a_3),
    $$
    $$
    \begin{aligned}
    F(2,11,a_3,a_4,\ldots, a_n) \le (2-1)(11-1) -1 &= 9 < 30.116122 \\ &= UB_2 (2,11,13) \le UB_2 (2,11,a_3),  
    \end{aligned}
    $$
    $$
    \begin{aligned}
    F(2,13,a_3,a_4,\ldots, a_n) \le (2-1)(13-1) -1 &= 11 < 39.083269 \\ &= UB_2 (2,13,15) \le UB_2 (2,13,a_3),
    \end{aligned}
    $$
    $$
    \begin{aligned}
    F(2,15,a_3,a_4,\ldots, a_n) \le (2-1)(15-1) -1 &= 13 < 48.8407169 \\ &= UB_2 (2,15,17) \le UB_2 (2,15,a3).
    \end{aligned}
    $$

    \item $a_1 = 3$:
    In this case, notice that 
    $$
    F(3,4,a_3,a_4,\ldots, a_n) \le (3-1)(4-1) -1 = 5 < 7.416408 = UB_2 (3,4,5) \le UB_2 (3,4,a_3), 
    $$
    $$
    F(3,5,a_3,a_4,\ldots, a_n) \le (3-1)(5-1) -1 = 9 < 12.343135 = UB_2 (3,5,7) \le UB_2 (3,5,a_3), 
    $$
    $$
    F(3,7,a_3,a_4,\ldots, a_n) \le (3-1)(7-1) -1 = 11 < 18.495454 = UB_2 (3,7,8) \le UB_2 (3,7,a_3), 
    $$
    $$
    F(3,8,a_3,a_4,\ldots, a_n) \le (3-1)(8-1) -1 = 13 < 27.105118 = UB_2 (3,8,11) \le UB_2 (3,8,a_3),
    $$
    $$
    \begin{aligned}
    F(3,10,a_3,a_4,\ldots, a_n) \le (3-1)(10-1) -1 &= 17 < 32.497191 \\ &= UB_2 (3,10,11) \le UB_2 (3,10,a_3) .
    \end{aligned}
    $$
    
    \item $a_1 = 4$:
    In this case, notice that 
    $$
    F(4,5,a_3,a_4,\ldots, a_n) \le (4-1)(5-1) -1 = 11 < 15.6643191 = UB_2 (4,5,7) \le UB_2 (4,5,a_3), 
    $$
    $$
    F(4,7,a_3,a_4,\ldots, a_n) \le (4-1)(7-1) -1 = 17 < 25.496479 = UB_2 (4,7,9) \le UB_2 (4,7,a_3).
    $$
    
    \item $a_1 = 5$:
    In this case, notice that 
    $$
    F(5,6,a_3,a_4,\ldots, a_n) \le (5-1)(6-1) -1 = 19 < 21.740852 = UB_2 (5,6,7) \le UB_2 (5,6,a_3).
    $$
\end{enumerate}
Thus, provided that $a_1 a_2 \le 32$, then $F(\boldsymbol{a}) < UB_2 (a_1, a_2, a_3) = UB_2 (\boldsymbol{a})$. Thus, it follows that $UB_2 (\boldsymbol{a})$ is indeed a valid upper bound in all cases, which concludes the proof.
\end{proof}

\section{Conclusions and Future Work}
In this paper, we revisited the classical Frobenius problem and examined a previously established upper bound on the Frobenius number. Our analysis revealed a subtle error in the original argument of Beck et al. \cite[Theorem 9]{beck2002frobenius}, leading to a revised and corrected upper bound. While this error did not invalidate the bound itself, it impacted its tightness. We also compared the relative tightness of the original and corrected bound through theoretical analysis and Monte Carlo simulations, demonstrating that the corrected bound is tighter in all but a relatively ``{small}'' (finite) number of cases.

This study opens several rather promising avenues for future research. Firstly, a case-specific analysis could explore the behavior and tightness of upper bounds under different assumptions around the distribution of the  input vector $\boldsymbol{a}$, such as uniform or exponential distributions. This would provide deeper insights into how the nature of the input impacts upon the bounds’ performance. 
Secondly, exploring applications of these bounds in optimization problems, particularly in knapsack or subset-sum problems, would be valuable. Understanding particularly how these bounds influence computational efficiency and solution ``quality'' in real-world settings could significantly broaden their utility. Finally, further examination of the geometric properties of the knapsack polytope $P(\boldsymbol{a},b)$ associated with the Frobenius problem may uncover deeper connections between geometric insights and the derivation of (perhaps) sharper upper bounds, particularly in higher-dimensional scenarios or under additional assumptions.

\end{document}